\documentclass[12pt]{article}
\usepackage{amssymb}
\usepackage{amsthm}
\usepackage{amsmath}
\usepackage{graphicx}
\usepackage{enumerate}

\DeclareMathOperator{\Con}{Con}
\DeclareMathOperator{\BCon}{\mathbf{Con}}

\newtheorem{theorem}{Theorem}[section]
\newtheorem{definition}[theorem]{Definition}

\newtheorem{remark}[theorem]{Remark}
\newtheorem{example}[theorem]{Example}
\newtheorem{corollary}[theorem]{Corollary}
\title{Join-semilattices whose principal filters are pseudocomplemented lattices}
\author{Ivan~Chajda and Helmut~L\"anger$^2$}
\date{}
\begin{document}

\footnotetext[1]{Support of the research by the Austrian Science Fund (FWF), project I~4579-N, and the Czech Science Foundation (GA\v CR), project 20-09869L, entitled ``The many facets of orthomodularity'', as well as by \"OAD, project CZ~02/2019, entitled ``Function algebras and ordered structures related to logic and data fusion'', and, concerning the first author, by IGA, project P\v rF~2021~030, is gratefully acknowledged.}

\footnotetext[2]{Corresponding author.}

\maketitle

\begin{abstract}
This paper deals with join-semilattices whose sections, i.e.\ principal filters, are pseudocomplemented lattices. The pseudocomplement of $a\vee b$ in the section $[b,1]$ is denoted by $a\rightarrow b$ and can be considered as the connective implication in a certain kind of intuitionistic logic. Contrary to the case of Brouwerian semilattices, sections need not be distributive lattices. This essentially allows possible applications in non-classical logics. We present a connection of the semilattices mentioned in the beginning with so-called non-classical implication semilattices which can be converted into I-algebras having everywhere defined operations. Moreover, we relate our structures to sectionally and relatively residuated semilattices which means that our logical structures are closely connected with substructural logics. We show that I-algebras form a congruence distributive, $3$-permutable and weakly regular variety.
\end{abstract}


non-classical implication semilattice, sectionally residuated join-semilattice, relatively residuated join-semilattice, I-algebra, variety of I-algebras, congruence distributive, 3-per\-mutable, weakly regular, R-algebra.

\section{Introduction}

The concept of implication algebra was introduced by J.~C.~Abbott (\cite{A67}) under the name semi-boolean algebra. It is in fact an axiomatization of the connective implication in classical logic, which was formalized by means of a Boolean algebra already by G.~Boole. As shown by Abbott, this algebra is in fact a join-semilattice with top element $1$ where every principal filter, the so-called {\em section} $[x,1]$, is a Boolean algebra. Hence, the operation $\vee$ is everywhere defined and it can be expressed by means of implication $x\cdot y$ as $x\vee y=(x\cdot y)\cdot y$. On the other hand, the operation $\wedge$ is only partial and $x\wedge y$ is defined if and only if the elements $x,y$ have a common lower bound $p$; in such a case $x\wedge y=\big(x\cdot(y\cdot p)\big)\cdot p$ and the complement of $x$ in $[p,1]$ is $x^p=x\cdot p$.

Later on, Abbott formalized in a similar way also the connective implication in the logic of quantum mechanics based on an orthomodular lattice, see \cite{A76}. In an analogous way, implication was axiomatized in ortholattices and in orthomodular lattices without the compatibility condition by the authors and R.~Hala\v s in \cite{CHL01} and \cite{CHL04}, see also \cite{C95} and \cite{C03b} for similar cases.

Concerning implication semilattices, i.e.\ join-semilattices where every section is a Boolean algebra, and concerning the structure and properties of such algebras, the reader is referred to \cite N, \cite{NW71} and \cite{NW73}.

It is a natural question whether the connective implication in intuitionistic logic can be axiomatized in a similar way. It is known that the so-called Brouwerian semilattices can serve for this sake, see e.g.\ \cite K. In this case we obtain a join-semilattice with top element $1$ whose sections are relatively pseudocomplemented posets or lattices, see \cite{Ba} for these concepts. However, relatively pseudocomplemented lattices are distributive (see e.g.\ \cite{Ba}), this fact restricting their application in non-classical logics.

An attempt to extend the concept of relative pseudocomplementation to non-distributive lattices was made by the first author in \cite{C03a}. This kind of pseudocomplementation was called {\em sectional pseudocomplementation} and it was extended also to posets with top element $1$ by the authors and J.~Paseka in \cite{CLP}.

Hence, we can extend our investigation of intuitionistic implication also to semilattices whose sections need not be distributive lattices. We show that in this case, these join-semilattices can be converted into the so-called sectionally residuated join-semilattices which are in one-to-one correspondence with relatively residuated join-semilattices introduced by the authors and J.~K\"uhr for lattices in \cite{CKL}.

Since in our structures the lattice operation $\wedge$ is only partial, the algebras under consideration are determined by so-called {\em partial identities} and hence they form the so-called {\em partial varieties}. To avoid drawbacks with partial operations, we introduce a new ternary everywhere defined operation and we show that our non-classical implication algebra can be converted in an algebra with everywhere defined operations characterized by identities only, and hence the corresponding class forms a variety. We show that this variety has nice congruence properties.

We believe that our approach brings a new insight in this topic and enables to understand the structure of investigated objects.

\section{Basic concepts}

A {\em pseudocomplemented lattice} (see e.g.\ \cite{Ba} and \cite{Bi}) is an algebra $(L,\vee,\wedge,{}^*,0)$ of type $(2,2,1,0)$ such that $(L,\vee,\wedge,0)$ is a lattice with bottom element $0$ and for every $x\in L$, $x^*$ is the greatest element $y\in L$ satisfying $x\wedge y=0$, the so-called {\em pseudocomplement of $x$}. It is easy to see that $x\leq y$ implies $y^*\leq x^*$ and that $x\leq x^{**}$.

Let $(S,\vee,1)$ be a join-semilattice with top element $1$ whose principal filters are pseudocomplemented lattices. This means that for every $x\in S$ there exists a binary operation $\wedge_x$ and a unary operation $^x$ on $[x,1]$ such that $([x,1],\vee,\wedge_x,{}^x)$ is a pseudocomplemented lattice with bottom element $x$. It is natural to assume the following {\em compatibility condition}:
\[
x\leq y\text{ and }z,u\in[y,1]\Rightarrow z\wedge_xu=z\wedge_yu.
\]
We then also have
\[
z,u\in[x,1]\cap[y,1]\Rightarrow z\wedge_xu=z\wedge_yu
\]
since $z,u\in[x,1]\cap[y,1]$ implies $z,u\in[x,1]\cap[y,1]\cap[x\vee y,1]$ and hence
\[
z\wedge_xu=z\wedge_{x\vee y}u=z\wedge_yu.
\]
In the following $z\wedge u$ denotes the infimum of $z$ and $u$ provided it exists. In general, the binary operation $\wedge$ is only partial. But now we can prove
\[
z,u\in[x,1]\Rightarrow z\wedge_xu=z\wedge u.
\]
This can be seen as follows: Because of $z,u\in[x,1]$ we have $z\wedge_xu\leq z,u$. If $y\in S$ and $y\leq z,u$ then $z,u\in[y,1]$ and hence $y\leq z\wedge_yu=z\wedge_xu$. This shows $z\wedge_xu=z\wedge u$. Hence $z\wedge u$ exists if and only if $z$ and $u$ have a common lower bound. This motivates the following definition.

\section{Join-semilattices with pseudocomplemented sec- \\
tions}

We start with the following fundamental definition.

\begin{definition}
A {\em join-semilattice whose sections are pseudocomplemented lattices} is an ordered quintuple $(S,\vee,\wedge,(^x;x\in S),1)$ where $(S,\vee,1)$ is a join-semilattice with top element $1$, $x\wedge y$ is defined if and only if $x$ and $y$ have a common lower bound and for every $x\in S$, $([x,1],\vee,\wedge,{}^x)$ is a pseudocomplemented lattice.
\end{definition}

Here and in the following the word ``{\em section}'' means the same as ``principal filter''.

In order to avoid problems with such a number of unary operations (pseudocomplementations in sections), we introduce a new binary operation $\rightarrow$ satisfying four simple axioms. This new structure would represent an implication reduct of a certain kind of non-classical (intuitionistic) logic. Because the name {\em implication semilattice} was already used by W.~C.~Nemitz (see e.g.\ \cite N, \cite{NW71} and \cite{NW73}) for a semilattice whose sections are Boolean algebras (which formalizes the implication reduct of classical logic), we use a different name for our structure. Hence, we define.

\begin{definition}\label{def1}
An {\em non-classical implication semilattice} is an ordered quintuple $(S,\vee,\wedge,$ $\rightarrow,1)$ where $(S,\vee,1)$ is a join-semilattice with top element $1$, $x\wedge y$ exists if and only if $x$ and $y$ have a common lower bound and $\rightarrow$ is a binary operation on $S$ satisfying {\rm(1)} -- {\rm(4)}:
\begin{enumerate}[{\rm(1)}]
\item $y\leq x\rightarrow y$,
\item $(x\vee y)\wedge(x\rightarrow y)\approx y$,
\item $(x\vee y)\rightarrow y\approx x\rightarrow y$,
\item $y\leq(x\vee z)\rightarrow\big((x\vee z)\wedge(y\vee z)\big)$.
\end{enumerate}
\end{definition}

\begin{remark}
Conditions {\rm(1)} and {\rm(3)} are identities whereas {\rm(2)} and {\rm(4)} can be considered as identities where the operation $\wedge$ is only partial. Hence the class of non-classical implication semilattices forms a partial variety.
\end{remark}

Observe that contrary to our case the sections of a relatively residuated join-semilattice need not be lattices and hence the meet of two elements which have a common lower bound need not exist.

For every join-semilattice $\mathbf S=(S,\vee,\wedge,(^x;x\in S),1)$ whose sections are pseudocomplemented lattices put $\mathbb I(\mathbf S):=(S,\vee,\wedge,\rightarrow,1)$ where $\rightarrow$ denotes the binary operation on $S$ defined by $x\rightarrow y:=(x\vee y)^y$ for all $x,y\in S$.

We can relate this concept with that one defined above.

\begin{theorem}\label{th1}
Let $\mathbf S=(S,\vee,\wedge,(^x;x\in S),1)$ be a join-semilattice whose sections are pseudocomplemented lattices. Then $\mathbb I(\mathbf S)=(S,\vee,\wedge,\rightarrow,1)$ is a non-classical implication semilattice satisfying {\rm(5)} -- {\rm(9)}:
\begin{enumerate}
\item[{\rm(5)}] $x\leq y\Leftrightarrow x\rightarrow y=1$,
\item[{\rm(6)}] $x\leq y\Rightarrow y\rightarrow z\leq x\rightarrow z$,
\item[{\rm(7)}] $x\leq(x\rightarrow y)\rightarrow y$,
\item[{\rm(8)}] $\big((x\rightarrow y)\rightarrow y\big)\rightarrow y\approx x\rightarrow y$,
\item[{\rm(9)}] $1\rightarrow x\approx x$.
\end{enumerate}
\end{theorem}

\begin{proof}
Let $a,b,c\in S$.
\begin{enumerate}[(1)]
\item $b\leq(a\vee b)^b=a\rightarrow b$
\item $(a\vee b)\wedge(a\rightarrow b)=(a\vee b)\wedge(a\vee b)^b=b$
\item $(a\vee b)\rightarrow b=\big((a\vee b)\vee b\big)^b=(a\vee b)^b=a\rightarrow b$
\item
Since $b\vee c\in[(a\vee c)\wedge(b\vee c),1]$ and $(a\vee c)\wedge(b\vee c)=(a\vee c)\wedge(b\vee c)$ we have
\[
b\vee c\leq(a\vee c)^{(a\vee c)\wedge(b\vee c)}
\]
and hence
\begin{align*}
b & \leq b\vee c\leq(a\vee c)^{(a\vee c)\wedge(b\vee c)}=\Big(a\vee c\vee\big((a\vee c)\wedge(b\vee c)\big)\Big)^{(a\vee c)\wedge(b\vee c)}= \\
  & =(a\vee c)\rightarrow\big((a\vee c)\wedge(b\vee c)\big)
\end{align*}
\item If $a\leq b$ then $a\rightarrow b=(a\vee b)^b=b^b=1$. If, conversely, $a\rightarrow b=1$ then $a\leq a\vee b=(a\vee b)\wedge1=(a\vee b)\wedge(a\rightarrow b)=(a\vee b)\wedge(a\vee b)^b=b$.
\item If $a\leq b$ then $b\rightarrow c=(b\vee c)^c\leq(a\vee c)^c=a\rightarrow c$.
\item $a\leq a\vee b\leq(a\vee b)^{bb}=\big((a\vee b)^b\vee b\big)^b=(a\rightarrow b)\rightarrow b$
\item By (7) and (6) we have
\[
a\rightarrow b\leq\big((a\rightarrow b)\rightarrow b\big)\rightarrow b\leq a\rightarrow b,
\]
i.e.
\[
\big((a\rightarrow b)\rightarrow b\big)\rightarrow b=a\rightarrow b.
\]
\item $1\rightarrow a=(1\vee a)^a=1^a=a$
\end{enumerate}
\end{proof}

The following example shows that in a join-semilattice whose sections are pseudocomplemented lattices these lattices need neither be distributive nor modular.

\begin{example}
The join-semilattice depictured in Figure~2

\vspace*{-2mm}

\begin{center}
\setlength{\unitlength}{7mm}
\begin{picture}(5,5)
\put(2,1){\circle*{.3}}
\put(1,2.5){\circle*{.3}}
\put(3,2){\circle*{.3}}
\put(3,3){\circle*{.3}}
\put(4,3){\circle*{.3}}
\put(2,4){\circle*{.3}}
\put(2,1){\line(-2,3)1}
\put(2,1){\line(1,1)1}
\put(2,4){\line(-2,-3)1}
\put(2,4){\line(1,-1)1}
\put(2,4){\line(2,-1)2}
\put(3,2){\line(0,1)1}
\put(1.85,.25){$0$}
\put(.3,2.35 ){$b$}
\put(3.4,1.85){$a$}
\put(3.4,2.85){$c$}
\put(4.4,2.85){$d$}
\put(1.85,4.4){$1$}
\put(1.2,-.75){{\rm Fig.\ 2}}
\end{picture}
\end{center}

\vspace*{4mm}

is a join-semilattice whose sections are pseudocomplemented lattices and the operation tables for $\rightarrow$ and $\wedge$ look as follows:
\[
\begin{array}{c|cccccc}
\rightarrow & 0 & a & b & c & d & 1 \\
\hline
     0      & 1 & 1 & 1 & 1 & 1 & 1\\
     a      & b & 1 & b & 1 & d & 1 \\
     b      & c & a & 1 & c & d & 1 \\
     c      & b & a & b & 1 & d & 1 \\
		 d      & 0 & a & b & c & 1 & 1 \\
     1      & 0 & a & b & c & d & 1
\end{array}
\quad\quad\quad
\begin{array}{c|cccccc}
\wedge & 0 & a & b & c & d & 1 \\
\hline
   0   & 0 & 0 & 0 & 0 & - & 0 \\
   a   & 0 & a & 0 & a & - & a \\
   b   & 0 & 0 & b & 0 & - & b \\
   c   & 0 & a & 0 & c & - & c \\
	 d   & - & - & - & - & d & d \\
   1   & 0 & a & b & c & d & 1
\end{array}
\]
\end{example}

\vspace*{10mm}

In what follows we show the converse way from a non-classical implication semilattice to a join-semilattice whose sections are pseudocomplemented lattices. For this sake, let us introduce the following concept.

For every non-classical implication semilattice $\mathbf I=(I,\vee,\wedge,\rightarrow,1)$ put
\[
\mathbb S(\mathbf I):=(I,\vee,\wedge,(^x;x\in I),1)
\]
where for every $x\in I$, $^x$ denotes the unary operation on $[x,1]$ defined by $y^x:=y\rightarrow x$ for all $y\in[x,1]$.

\begin{theorem}\label{th2}
Let $\mathbf I=(I,\vee,\wedge,\rightarrow,1)$ be a non-classical implication semilattice. Then $\mathbb S(\mathbf I)$ is a join-semi\-lat\-tice whose sections are pseudocomplemented lattices.
\end{theorem}

\begin{proof}
Let $\mathbb S(\mathbf I)=(I,\vee,\wedge,(^x;x\in I),1)$ and $a,b,c\in I$ with $c\leq a,b$. If $a\wedge b=c$ then according to (4) we obtain
\[
b\leq(a\vee c)\rightarrow\big((a\vee c)\wedge(b\vee c)\big)=a\rightarrow(a\wedge b)=a\rightarrow c=a^c.
\]
If, conversely, if $b\leq a^c$ then
\[
c\leq a\wedge b\leq a\wedge a^c\leq(a\vee c)\wedge(a\rightarrow c)=c
\]
according to (2) and hence $a\wedge b=c$. Hence, for each $a\in[c,1]$ the element $a\rightarrow c$ is the pseudocomplement of $a$ in this interval.
\end{proof}

We justify our conversions by the following result.

\begin{theorem}
The correspondence described in Theorems~\ref{th1} and \ref{th2} is one-to-one.
\end{theorem}

\begin{proof}
If $\mathbf S=\big(S,\vee,\wedge,(^x;x\in S),1\big)$ is a join-semilattice whose sections are pseudocomplemented lattices and
\begin{align*}
                   \mathbb I(\mathbf S) & =(S,\vee,\wedge,\rightarrow,1), \\
\mathbb S\big(\mathbb I(\mathbf S)\big) & =\big(S,\vee,\wedge,(_x;x\in S),1\big)
\end{align*}
then $y_x=y\rightarrow x=(y\vee x)^x=y^x$ for all $x\in S$ and $y\in[x,1]$ and hence $\mathbb S\big(\mathbb I(\mathbf S)\big)=\mathbf S$. If, conversely, $\mathbf I=(I,\vee,\wedge,\rightarrow,1)$ is a non-classical implication semilattice and
\begin{align*}
                   \mathbb S(\mathbf I) & =\big(I,\vee,\wedge,(^x;x\in S),1\big), \\
\mathbb I\big(\mathbb S(\mathbf I)\big) & =(I,\vee,\wedge,\Rightarrow,1)
\end{align*}
then $x\Rightarrow y\approx(x\vee y)^y\approx(x\vee y)\rightarrow y\approx x\rightarrow y$ according to (3) and hence $\mathbb I\big(\mathbb S(\mathbf I)\big)=\mathbf I$ where $(x\vee y)\Rightarrow y\approx\big((x\vee y)\vee y\big)^y\approx(x\vee y)^y$.
\end{proof}

\section{Sectionally and relatively residuated join-semi\-lat\-tices}

The aim of this section is to establish a relationship between the structures defined in the previous section and residuated structures which formalize certain substructural logics.

We start with definitions of our key concepts relating our implication algebras with a certain kind of residuation.

\begin{definition}\label{def3}
A {\em sectionally residuated join-semilattice} is an ordered quintuple $(S,\vee,$ $(\odot_x;x\in S),\rightarrow,1)$ such that $(S,\vee,1)$ is a join-semlattice with top element $1$, for each $x\in S$, $([x,1],\odot_x,1)$ is a commutative monoid, $\rightarrow$ is a binary operation on $S$ and for all $x,y,z,u\in S$ the following hold:
\begin{enumerate}[{\rm(i)}]
\item if $z\leq u\leq x,y$ then $x\odot_zy=x\odot_uy$ {\rm(}{\em compatibility condition}{\rm)},
\item if $u\leq x,y,z$ and $x\leq y$ then $x\odot_uz\leq y\odot_uz$ {\rm(}{\em monotonicity}{\rm)},
\item $(x\vee z)\odot_z(y\vee z)\leq z$ if and only if $x\vee z\leq y\rightarrow z$ {\rm(}{\em sectional adjointness}{\rm)},
\item $(x\vee y)\rightarrow y\approx x\rightarrow y$.
\end{enumerate}
\end{definition}

The next concept we borrow from our previous paper \cite{CKL} where, however, it is introduced for lattices and hence the meet-operation is everywhere defined in the original definition.

\begin{definition}\label{def2}
A {\em relatively residuated join-semilattice} is a partial algebra $\mathbf R=(R,\vee,$ $\odot,$ $\rightarrow,1)$ of type $(2,2,2,0)$ where $x\odot y$ is defined if and only if $x$ and $y$ have a common lower bound and in this case $z\leq x\odot y$ for all $z\in R$ with $z\leq x,y$ and $\rightarrow$ is a binary operation on $R$ such that for all $x,y,z\in R$ the following hold:
\begin{enumerate}
\item[{\rm(10)}] $(R,\vee,1)$ is a join-semilattice with top element $1$,
\item[{\rm(11)}] $x\odot1\approx1\odot x\approx x$,
\item[{\rm(12)}] $x\odot y=y\odot x$ whenever $x$ and $y$ have a common lower bound,
\item[{\rm(13)}] $(x\odot y)\odot z=x\odot(y\odot z)$ whenever $x$, $y$ and $z$ have a common lower bound,
\item[{\rm(14)}] $x\leq y$ implies $x\odot z\leq y\odot z$ whenever $x$ and $z$ have a common lower bound {\rm(}{\em monotonicity}{\rm)},
\item[{\rm(15)}] $(x\vee z)\odot(y\vee z)\leq z$ if and only if $x\vee z\leq y\rightarrow z$ {\rm(}{\em relative adjointness}{\rm)},
\item[{\rm(16)}] $(x\vee y)\rightarrow y\approx x\rightarrow y$.
\end{enumerate}
$\mathbf R$ is called {\em divisible} if it satisfies the identity
\[
(x\vee y)\odot(x\rightarrow y)\approx y.
\]
\end{definition}

In the next theorem we describe several basic properties of relatively residuated join-semilattices.

\begin{theorem}\label{th6}
Let $\mathbf S=(S,\vee,\odot,\rightarrow,1)$ be a relatively residuated join-semilattice and $a,b,c\in S$. Then the following hold:
\begin{enumerate}[{\rm(i)}]
\item $a\leq b\Leftrightarrow a\rightarrow b=1$,
\item $a\odot b\leq a$,
\item if $a\wedge b$ exists then $a\odot b\leq a\wedge b$,
\item $a\leq b\rightarrow a$,
\item $a\odot(a\rightarrow b)\leq(a\vee b)\odot(a\rightarrow b)\leq b$,
\item $a\leq(a\rightarrow b)\rightarrow b$,
\item $a\leq b\Rightarrow b\rightarrow c\leq a\rightarrow c$,
\item $\big((a\rightarrow b)\rightarrow b\big)\rightarrow b=a\rightarrow b$.
\end{enumerate}
\end{theorem}

\begin{proof}
\
\begin{enumerate}[(i)]
\item The following are equivalent:
\begin{align*}
                      a & \leq b, \\
(1\vee b)\odot(a\vee b) & \leq b, \\
                1\vee b & \leq a\rightarrow b, \\
         a\rightarrow b & =1.
\end{align*}
\item By commutativity and monotonicity of $\odot$ we have $a\odot b=b\odot a\leq1\odot a=a$.
\item This follows from (ii) and from commutativity of $\odot$.
\item We have $(a\vee b)\odot a\leq a$ according to (ii). Now the following are equivalent:
\begin{align*}
(a\vee a)\odot(b\vee a) & \leq a, \\
                a\vee a & \leq b\rightarrow a, \\
                      a & \leq b\rightarrow a.
\end{align*}
\item We have $b\leq a\rightarrow b$ according to (iv). Now by (iv) the following are equivalent:
\begin{align*}
                        (a\rightarrow b)\vee b & \leq a\rightarrow b, \\
\big((a\rightarrow b)\vee b\big)\odot(a\vee b) & \leq b, \\
                (a\vee b)\odot(a\rightarrow b) & \leq b.
\end{align*}
By monotonicity of $\odot$ we have $a\odot(a\rightarrow b)\leq(a\vee b)\odot(a\rightarrow b)$.
\item We have $(a\vee b)\odot(a\rightarrow b)\leq b$ according to (v). Now by (iv) the following are equivalent:
\begin{align*}
(a\vee b)\odot\big((a\rightarrow b)\vee b\big) & \leq b, \\
                                       a\vee b & \leq(a\rightarrow b)\rightarrow b.
\end{align*}
\item Assume $a\leq b$. Then $a\vee c\leq(b\rightarrow c)\rightarrow c$ according to (iv) and (vi). Now the following are equivalent:
\begin{align*}
                                       a\vee c & \leq(b\rightarrow c)\rightarrow c, \\
(a\vee c)\odot\big((b\rightarrow c)\vee c\big) & \leq c, \\
\big((b\rightarrow c)\vee c\big)\odot(a\vee c) & \leq c, \\
                        (b\rightarrow c)\vee c & \leq a\rightarrow c, \\
                                b\rightarrow c & \leq a\rightarrow c.
\end{align*}
\item According to (vi) and (vii) we have
\[
a\rightarrow b\leq\big((a\rightarrow b)\rightarrow b\big)\rightarrow b\leq a\rightarrow b
\]
proving (viii).
\end{enumerate}
\end{proof}

In next theorem we show how the class of relatively residuated join-semilattices can be described without using of relative adjointness.

\begin{theorem}\label{th11}
Let $\mathbf R=(R,\vee,$ $\odot,$ $\rightarrow,1)$ be a partial algebra of type $(2,2,2,0)$ where $x\odot y$ is defined if and only if $x$ and $y$ have a common lower bound and in this case $z\leq x\odot y$ for all $z\in R$ with $z\leq x,y$ and $\rightarrow$ is a binary operation on $R$ such that for all $x,y,z\in R$ {\rm(10)} -- {\rm(14)} and {\rm(16)} hold. Then $\mathbf R$ is a relatively residuated join-semilattice if and only if it satisfies the following identities:
\begin{enumerate}
\item[{\rm(17)}] $x\vee z\leq y\rightarrow\Big(\big((x\vee z)\odot(y\vee z)\big)\vee z\Big)$,
\item[{\rm(18)}] $x\leq y\rightarrow x$,
\item[{\rm(19)}] $(x\vee y)\odot(x\rightarrow y)\leq y$.
\end{enumerate}
\end{theorem}

\begin{proof}
Let $a,b,c,d\in R$. If $\mathbf R$ is a relatively residuated join-semilattice then (18) and (19) follow from Theorem~\ref{th6}, and if $d:=\big((a\vee c)\odot(b\vee c)\big)\vee c$ then $a\vee d=a\vee c$ and $b\vee d=b\vee c$ because of (ii) of Theorem~\ref{th6} and (12) whence
\[
(a\vee d)\odot(b\vee d)=(a\vee c)\odot(b\vee c)\leq d
\]
by definition of $d$ which implies
\[
a\vee c=a\vee d\leq b\rightarrow d=b\rightarrow\Big(\big((a\vee c)\odot(b\vee c)\big)\vee c\Big)
\]
by (15) showing (17). If, conversely, $\mathbf R$ satisfies (17) -- (19) then, if $(a\vee c)\odot(b\vee c)\leq c$ then $a\vee c\leq b\rightarrow c$ according to (17), and if, conversely, $a\vee c\leq b\rightarrow c$ then
\[
(a\vee c)\odot(b\vee c)\leq(b\rightarrow c)\odot(b\vee c)=(b\vee c)\odot(b\rightarrow c)\leq c
\]
according to (14), (12) and (19) proving (15), thus showing that $\mathbf R$ is a relatively residuated join-semilattice.
\end{proof}

Theorem~\ref{th11} implies that the class of relatively residuated join-semilattices is determined by partial identities and hence forms a partial variety.
 
The following results show that the concepts just defined can be converted into each other.

\begin{theorem}\label{th5}
\
\begin{enumerate}[{\rm(i)}]
\item Let $\mathbf S=(S,\vee,(\odot_x;x\in S),\rightarrow,1)$ be a sectionally residuated join-semilattice and define a partial binary operation $\odot$ on $S$ as follows: If $x$ and $y$ have a common lower bound $z$ then $x\odot y:=x\odot_zy$. Then $\odot$ is well-defined and $\mathbb R(\mathbf S):=(S,\vee,\odot,\rightarrow,1)$ is a relatively residuated join-semilattice.
\item Let $\mathbf R=(R,\vee,\odot,\rightarrow,1)$ be a relatively residuated join-semilattice and put $\odot_x:=\odot|[x,1]$ for all $x\in S$. Then $\mathbb S(\mathbf R):=(R,\vee,(\odot_x;x\in R),\rightarrow,1)$ is a sectionally residuated join-semilattice.
\end{enumerate}
\end{theorem}

\begin{proof}
(16) coincides with (iv) of Definition~\ref{def3}.
\begin{enumerate}[(i)]
\item Let $a,b,c,d\in S$. If $c$ and $d$ are common lower bounds of $a$ and $b$ then also $c\vee d$ is a common lower bound of $a$ and $b$ and hence $a\odot_cb=a\odot_{c\vee d}b=a\odot_db$. This shows that $\odot$ is well-defined. Moreover, we have $a\odot1=a\odot_a1=a$ and $1\odot a=1\odot_aa=a$. If $a$ and $b$ have a common lower bound $c$ then $a\odot b=a\odot_cb=b\odot_ca=b\odot a$. If, finally, $a$, $b$ and $c$ have a common lower bound $d$ then
\[
(a\odot b)\odot c=(a\odot_db)\odot_dc=a\odot_d(b\odot_dc)=a\odot(b\odot c)
\]
and $a\leq b$ implies $a\odot c=a_dc\leq b\odot_dc=b\odot c$. Relative adjointness follows from $(a\vee c)\odot(b\vee c)=(a\vee c)\odot_c(b\vee c)$.
\item Let $a,b,c,d\in R$. Then $([a,1],\odot_a,1)$ is a commutative monoid. Moreover, if $c\leq d\leq a,b$ then $a\odot_cb=a\odot b=a\odot_db$. Further, if $d\leq a,b,c$ and $a\leq b$ then $a\odot_dc=a\odot c\leq b\odot c=b\odot_dc$. Finally, sectional adjointness follows from $(a\vee c)\odot_c(b\vee c)=(a\vee c)\odot(b\vee c)$.
\end{enumerate}
\end{proof}

The next result shows that the conversions described in Theorem~\ref{th5} are nothing else then a translation from one language into the other.

\begin{theorem}
The correspondence described in Theorem~\ref{th5} is one-to-one.
\end{theorem}

\begin{proof}
If $\mathbf S=(S,\vee,(\odot_x;x\in S),\rightarrow,1)$ is a sectionally residuated join-semilattice,
\begin{align*}
                   \mathbb R(\mathbf S) & =(S,\vee,\odot,\rightarrow,1), \\
\mathbb S\big(\mathbb R(\mathbf S)\big) & =(S,\vee,(\circ_x;x\in S),\rightarrow,1),
\end{align*}
$a,b,c\in S$ and $c\leq a,b$ then $a\odot_cb=a\odot b=a\odot_cb$ showing $\mathbb S\big(\mathbb R(\mathbf S)\big)=\mathbf S$. If, conversely $\mathbf R=(R,\vee,\odot,\rightarrow,1)$ is a relatively residuated join-semilattice,
\begin{align*}
                   \mathbb S(\mathbf R) & =(R,\vee,(\odot_x;x\in R),\rightarrow,1), \\
\mathbb R\big(\mathbb S(\mathbf R)\big) & =(R,\vee,\circ,\rightarrow,1),
\end{align*}
$a,b,c\in R$ and $c\leq a,b$ then $a\circ b=a\odot_cb=a\odot b$ showing $\mathbb R\big(\mathbb S(\mathbf R)\big)=\mathbf R$.
\end{proof}

In the next result we show that non-classical implication semilattices form a partial subvariety of the class of divisible relatively residuated join-semilattices.

According to (iii) of Theorem~\ref{th6} $x\odot y\leq x\wedge y$ whenever $x\wedge y$ exists. The question when $x\odot y=x\wedge y$ is answered in the next theorem; this equality turns out to hold if $\odot$ is idempotent (and satisfies one more simple identity ).

\begin{theorem}\label{th10}
The non-classical implication semilattices are exactly the divisible relatively residuated join-semilattices $(R,\vee,\odot,\rightarrow,1)$ satisfying the following identities:
\begin{enumerate}[{\rm(i)}]
\item $x\odot x\approx x$,
\item $y\leq(x\vee z)\rightarrow\big((x\vee z)\odot(y\vee z)\big)$.
\end{enumerate}
\end{theorem}

\begin{proof}
Let $(S,\vee,\wedge,\rightarrow,1)$ be a non-classical implication semilattice. Then $x\wedge y$ is defined if and only if $x$ and $y$ have a common lower bound, $z\leq x\wedge y$ whenever $z\leq x,y$, $\rightarrow$ is a binary operation on $S$ and (10) -- (14) of Definition~\ref{def2} hold. Now let $a,b,c\in S$. First assume $(a\vee c)\wedge(b\vee c)\leq c$ Then $(b\vee c)\wedge(a\vee c)=c$. Since
\[
a\leq(b\vee c)\rightarrow\big((b\vee c)\wedge(a\vee c)\big)=(b\vee c)\rightarrow c=b\rightarrow c
\]
according to (4) and (3) and $c\leq b\rightarrow c$ according to (1), we conclude $a\vee c\leq b\rightarrow c$. If, conversely, $a\vee c\leq b\rightarrow c$ then
\[
(a\vee c)\wedge(b\vee c)=(b\vee c)\wedge(a\vee c)\leq(b\vee c)\wedge(b\rightarrow c)=c
\]
by (2). This shows relative adjointness. (16) coincides with (3). Divisibility of $\mathbf S$ follows from (2). Obviously, (i) holds. Moreover, (ii) follows from (3), and (iii) follows from (4). Conversely, let $(R,\vee,\odot,\rightarrow,1)$ be a divisible relatively residuated join-semilattice satisfying identities (i) and (ii). Then $(R,\vee,1)$ is a join-semilattice with top element $1$ according to (10), $x\odot y$ is defined if and only if $x$ and $y$ have a common lower bound and $\rightarrow$ is a binary operation on $R$. Let $a,b,c,d\in R$ with $c\leq a,b$. Because of (14), (11) and (12) we have $a\odot b\leq1\odot b=b$ and $a\odot b=b\odot a\leq a$. If, moreover, $d\leq a,b$ then because of (i), (14) and (12) we have
\[
d=d\odot d\leq a\odot d=d\odot a\leq b\odot a=a\odot b.
\]
This shows $a\odot b=a\wedge b$. Finally, (1) follows from (iv) of Theorem~\ref{th6}, (2) from divisibility, (3) from (16) and (4) from (ii).
\end{proof}

\begin{remark}
In Theorem~\ref{th10} identity {\rm(i)} can be replaced by condition {\rm(i')}
\begin{enumerate}[{\rm(i')}]
\item $x\wedge y$ exists if and only if $x$ and $y$ have a common lower bound.
\end{enumerate}
This can be seen as follows: Let $(R,\vee,\odot,\rightarrow,1)$ be a relatively residuated join-semilattice satisfying {\rm(i')} and $a,b\in R$ having a common lower bound. Then $a\wedge b$ exists by {\rm(i')} and $a\wedge b$ is a common lower bound of $a$ and $b$ whence $a\wedge b\leq a\odot b$. But $a\odot b\leq a\wedge b$ by {\rm(iii)} of Theorem~\ref{th6} which shows $a\odot b=a\wedge b$.
\end{remark}

\begin{example}
The join-semilattice depicted in Figure~1

\vspace*{-2mm}

\begin{center}
\setlength{\unitlength}{7mm}
\begin{picture}(6,3)
\put(1,0){\circle*{.3}}
\put(2,1){\circle*{.3}}
\put(3,2){\circle*{.3}}
\put(4,1){\circle*{.3}}
\put(5,0){\circle*{.3}}
\put(3,2){\line(-1,-1)2}
\put(3,2){\line(1,-1)2}
\put(.3,-.15){$a$}
\put(1.3,.85){$b$}
\put(2.85,2.4){$1$}
\put(5.4,-.15){$c$}
\put(4.4,.85){$d$}
\put(2.2,-.75){{\rm Fig.\ 1}}
\end{picture}
\end{center}

\vspace*{4mm}

is a join-semilattice whose sections are pseudocomplemented lattices and the operation tables for $\rightarrow$ and $\wedge$ look as follows:
\[
\begin{array}{c|ccccc}
\rightarrow & a & b & c & d & 1 \\
\hline
     a      & 1 & 1 & c & d & 1 \\
     b      & a & 1 & c & d & 1 \\
     c      & a & b & 1 & 1 & 1 \\
     d      & a & b & c & 1 & 1 \\
     1      & a & b & c & d & 1
\end{array}
\quad\quad\quad
\begin{array}{c|ccccc}
\wedge & a & b & c & d & 1 \\
\hline
   a   & a & a & - & - & a \\
   b   & a & b & - & - & b \\
   c   & - & - & c & c & c \\
   d   & - & - & c & d & d \\
   1   & a & b & c & d & 1
\end{array}
\]
It is easy to check adjointness in the sense that whenever $x\wedge y$ is defined we have $x\wedge y\leq z$ if and only $x\leq y\rightarrow z$. It should be mentioned that $\rightarrow$ is a proper extension of the pseudocomplementation within sections. Namely, the pseudocomplement of $b$ within $[a,1]$ is $b\rightarrow a=a$. On the other hand, $b\notin[d,1]$, but $b\rightarrow d=d$.
\end{example}

\section{The variety of non-classical implication algebras}

Up to now, we investigated algebras where the operations $\vee$ and $\rightarrow$ were everywhere defined, but $\wedge$ or $\odot$ were only partial operations defined just in the case when both the operands belong to the same section. To avoid this difficulty, we introduce a new ternary operation $r(x,y,z)$ which is everywhere defined and replaces the partial operation $\wedge$.

\begin{definition}
An {\em {\rm I}-algebra} is an algebra $(A,\vee,r,\rightarrow,1)$ of type $(2,3,2,0)$ such that $(A,\vee,1)$ is a join-semilattice with top element $1$ and the following identities are satisfied:
\begin{enumerate}[{\rm(1')}]
\item $y\leq x\rightarrow y$,
\item $r(x,x\rightarrow y,y)\approx y$,
\item $(x\vee y)\rightarrow y\approx x\rightarrow y$,
\item $y\leq(x\vee z)\rightarrow r(x,y,z)$,
\item $r(x,y,z)\leq x\vee z$,
\item $r(x,y,z)\leq y\vee z$,
\item $r(x,x\vee y,z)\approx x\vee z$,
\item $r(x,y,z)\approx r(x\vee z,y\vee z,z)$,
\item $z\leq r(x,y,z)$,
\item $r(u,r(x,y,z),z)\approx r(r(u,x,z),r(u,y,z),z)$.
\end{enumerate}
\end{definition}

Considering a non-classical implication semilattice as introduced in Definition~\ref{def1}, we can define a ternary operation $r(x,y,z)$ as follows: $r(x,y,z):=(x\vee z)\wedge(y\vee z)$. For every non-classical implication semilattice $\mathbf I=(I,\vee,\wedge,\rightarrow,1)$ put $\mathbb A(\mathbf I):=(I,\vee,r,\rightarrow,1)$.

The next theorem shows a conversion of non-classical implication join-semilattices into I-algebras.

\begin{theorem}\label{th3}
Let $\mathbf I=(I,\vee,\wedge,\rightarrow,1)$ be a non-classical implication semilattice. Then $\mathbb A(\mathbf I)$ is an {\rm I}-algebra.
\end{theorem}

\begin{proof}
Let $\mathbb A(\mathbf I)=(I,\vee,r,\rightarrow,1)$ and $a,b,c,d\in I$. Then, clearly, $r(x,x,y)\approx x\vee y\approx r(y,y,x)$.
\begin{enumerate}[(1')]
\item coincides with (1).
\item $r(a,a\rightarrow b,b)=(a\vee b)\wedge\big((a\rightarrow b)\vee b\big)=(a\vee b)\wedge(a\rightarrow b)=b$ according to (1) and (2).
\item coincides with (3).
\item $b\leq(a\vee c)\rightarrow\big((a\vee c)\wedge(b\vee c)\big)=(a\vee c)\rightarrow r(a,b,c)$ according to (4).
\item $r(a,b,c)=(a\vee c)\wedge(b\vee c)\leq a\vee c$
\item $r(a,b,c)=(a\vee c)\wedge(b\vee c)\leq b\vee c$
\item $r(a,a\vee b,c)=(a\vee c)\wedge\big((a\vee b)\vee c\big)=a\vee c$
\item $r(a,b,c)=(a\vee c)\wedge(b\vee c)=\big((a\vee c)\vee c\big)\wedge\big((b\vee c)\vee c\big)=r(a\vee c,b\vee c,c)$
\item $c\leq(a\vee c)\wedge(b\vee c)=r(a,b,c)$
\item
\begin{align*}
r(d,r(a,b,c),c) & =(d\vee c)\wedge\Big(\big((a\vee c)\wedge(b\vee c)\big)\vee c\Big)=(d\vee c)\wedge\big((a\vee c)\wedge(b\vee c)\big)= \\
                & =(d\vee c)\wedge(a\vee c)\wedge(b\vee c)= \\
								& =\big((d\vee c)\wedge(a\vee c)\big)\wedge\big((d\vee c)\wedge(b\vee c)\big)= \\
                & =\Big(\big((d\vee c)\wedge(a\vee c)\big)\vee c\Big)\wedge\Big(\big((d\vee c)\wedge(b\vee c)\big)\vee c\Big)= \\
                & =r\big(r(d,a,c),r(d,b,c),c\big)
\end{align*}
\end{enumerate}
\end{proof}

For every I-algebra $\mathbf A=(A,\vee,r,\rightarrow,1)$ put $\mathbb J(\mathbf A):=(A,\vee,\wedge,\rightarrow,1)$ where $x\wedge y:=r(x,y,z)$ for all $x,y,z\in A$ with $z\leq x,y$. That $\mathbb J(\mathbf A)$ is well-defined follows from the next theorem.

\begin{theorem}\label{th4}
Let $\mathbf A=(A,\vee,r,\rightarrow,1)$ be an {\rm I}-algebra. Then $\mathbb J(\mathbf A)$ is a well-defined non-classical implication semilattice.
\end{theorem}

\begin{proof}
Let $\mathbb J(\mathbf A)=(A,\vee,\wedge,\rightarrow,1)$ and $a,b,c,d\in A$. First we prove $r(a,b,c)=(a\vee c)\wedge(b\vee c)$. Because of (5') and (6') we have $r(a,b,c)\leq a\vee c,b\vee c$. Now assume $d\leq a\vee c,b\vee c$. Then
\begin{align*}
d & \leq d\vee c=(d\vee c)\vee c=r\big(d\vee c,(d\vee c)\vee c,c\big)=r\big((d\vee c)\vee c,(d\vee c)\vee c,c\big)= \\
  & =r\Big(r\big(d\vee c,(d\vee c)\vee(a\vee c),c\big),r\big(d\vee c,(d\vee c)\vee(b\vee c),c\big),c\Big)= \\
  & =r\big(r(d\vee c,a\vee c,c),r(d\vee c,b\vee c,c),c\big)=r\big(r(d,a,c),r(d,b,c),c\big)=r\big(d,r(a,b,c),c\big)\leq \\
  & \leq r(a,b,c)\vee c=r(a,b,c)
\end{align*}
because of (7'), (8'), (10'), (6') and (9'). This shows $r(a,b,c)=(a\vee c)\wedge(b\vee c)$. Hence $a\wedge b=\inf(a,b)$ whenever there exists a common lower bound of $a$ and $b$ and therefore $\wedge$ and $\mathbb J(\mathbf A)$ are well-defined.
\begin{enumerate}[(1)]
\item coincides with (1').
\item $(a\vee b)\wedge(a\rightarrow b)=r(a,a\rightarrow b,b)=b$ because of (1') and (2').
\item coincides with (3').
\item $b\leq(a\vee c)\rightarrow r(a,b,c)=(a\vee c)\rightarrow\big((a\vee c)\wedge(b\vee c)\big)$ by (4').
\end{enumerate}
\end{proof}

\begin{theorem}
The correspondence described in Theorems~\ref{th3} and \ref{th4} is one-to-one.
\end{theorem}

\begin{proof}
If $\mathbf I=(I,\vee,\wedge,\rightarrow,1)$ is a non-classical implication semilattice and
\begin{align*}
                   \mathbb A(\mathbf I) & =(I,\vee,r,\rightarrow,1), \\
\mathbb J\big(\mathbb A(\mathbf I)\big) & =(I,\vee,\cap,\rightarrow,1)
\end{align*}
then $x\cap y=r(x,y,z)=(x\vee z)\wedge(y\vee z)=x\wedge y$ for all $x,y,z\in I$ with $z\leq x,y$ and hence $\mathbb J\big(\mathbb A(\mathbf I)\big)=\mathbf I$. If, conversely, $\mathbf A=(A,\vee,r,\rightarrow,1)$ is an I-algebra and
\begin{align*}
                   \mathbb J(\mathbf A) & =(A,\vee,\wedge,\rightarrow,1), \\
\mathbb A\big(\mathbb J(\mathbf A)\big) & =(A,\vee,R,\rightarrow,1)
\end{align*}
then $R(x,y,z)\approx(x\vee z)\wedge(y\vee z)\approx r(x\vee z,y\vee z,z)\approx r(x,y,z)$ according to (8') and hence $\mathbb A\big(\mathbb J(\mathbf A)\big)=\mathbf A$.
\end{proof}

The advantage of using a ternary operation $r(x,y,z)$ instead of the partial binary operation $\wedge$ is that now we have an algebra with everywhere defined operations which is determined by a finite set of identities. Hence these algebras form a variety $\mathcal V$.

Recall that an {\em algebra} $\mathbf A$ is called
\begin{itemize}
\item {\em $3$-permutable} if $\Theta\circ\Phi\circ\Theta=\Phi\circ\Theta\circ\Phi$ for all $\Theta,\Phi\in\Con\mathbf A$,
\item {\em congruence distributive} if the congruence lattice $\BCon\mathbf A$ of $\mathbf A$ is distributive.
\end{itemize}
An {\em algebra} $\mathbf A$ with an equationally definable constant $1$ is called {\em weakly regular} if $\Theta,\Phi\in\Con\mathbf A$ and $[1]\Theta=[1]\Phi$ imply $\Theta=\Phi$. A {\em variety} is called {\em $3$-permutable} or {\em congruence distributive} if every of its members has the respective property. A {\em variety} with an equationally definable constant $1$ is called {\em weakly regular} if every of its members has this property. A variety is $3$-permutable (cf.\ \cite{CEL}, Theorem~3.1.18) if and only if there exist ternary terms $t_1,t_2$ satisfying
\begin{align*}
t_1(x,y,y) & \approx x, \\
t_1(x,x,y) & \approx t_2(x,y,y), \\
t_2(x,x,y) & \approx y.
\end{align*}
A variety is congruence distributive (cf.\ \cite{CEL}, Theorem~3.2.2) if and only if there exists a positive integer $n$ and there exist ternary terms $t_0,\ldots,t_n$ (so-called {\em J\'onsson terms}) satisfying
\begin{align*}
t_0(x,y,z) & \approx x, \\
t_i(x,x,y) & \approx t_{i+1}(x,x,y)\text{ for all even }i\in\{0,\ldots,n-1\}, \\
t_i(x,y,y) & \approx t_{i+1}(x,y,y)\text{ for all odd }i\in\{1,\ldots,n-1\}, \\
t_n(x,y,z) & \approx z, \\
t_i(x,y,x) & \approx x\text{ for all }i\in\{1,\ldots,n-1\}.
\end{align*}
A variety with an equationally definable constant $1$ is weakly regular (cf.\ \cite{CEL}, Theorem~6.4.3) if and only if there exists a positive integer $n$ and there exist binary terms $t_1,\ldots,t_n$ such that $t_1(x,y)=\cdots=t_n(x,y)=1$ is equivalent to $x=y$.

\begin{theorem}\label{th12}
The variety $\mathcal V$ of {\rm I}-algebras is $3$-permutable, congruence distributive and weakly regular.
\end{theorem}

\begin{proof}
Let $\mathbf A=(A,\vee,r,\rightarrow,1)$ be an I-algebra and $\mathbb J(\mathbf A)=(A,\vee,\wedge,\rightarrow,1)$. \\
If we put
\begin{align*}
t_1(x,y,z) & :=r(z,y\rightarrow x,x), \\
t_2(x,y,z) & :=r(x,y\rightarrow z,z)
\end{align*}
then
\begin{align*}
t_1(x,y,z) & \approx(z\vee x)\wedge\big((y\rightarrow x)\vee x\big)\approx(z\vee x)\wedge\big(y\rightarrow x)\text{ by (1)}, \\
t_2(x,y,z) & \approx(x\vee z)\wedge\big((y\rightarrow z)\vee z\big)\approx(x\vee z)\wedge\big(y\rightarrow z)\text{ by (1)}, \\
t_1(x,y,y) & \approx(y\vee x)\wedge(y\rightarrow x)\approx x\text{ by (2)}, \\
t_1(x,x,y) & \approx(y\vee x)\wedge\big(x\rightarrow x)\approx(x\vee y)\wedge1\approx(x\vee y)\wedge\big(y\rightarrow y)\approx t_2(x,y,y)\text{ by (5)}, \\
t_2(x,x,y) & \approx(x\vee y)\wedge\big(x\rightarrow y)\approx y\text{ by (2)}.
\end{align*}
and hence $\mathcal V$ is $3$-permutable. \\
If we put
\begin{align*}
t_0(x,y,z) & :=x, \\
t_1(x,y,z) & :=r(z,y,x), \\
t_2(x,y,z) & :=r(x,y\rightarrow z,z), \\
t_3(x,y,z) & :=z
\end{align*}
then
\begin{align*}
t_1(x,y,z) & \approx(z\vee x)\wedge(y\vee x), \\
t_2(x,y,z) & \approx(x\vee z)\wedge\big((y\rightarrow z)\vee z\big)\approx(x\vee z)\wedge(y\rightarrow z)\text{ by (1)}, \\
t_0(x,x,y) & \approx x\approx(x\vee y)\wedge x\approx(y\vee x)\wedge(x\vee x)\approx t_1(x,x,y), \\
t_1(x,y,y) & \approx(y\vee x)\wedge(y\vee x)\approx x\vee y\approx(x\vee y)\wedge1\approx(x\vee y)\wedge(y\rightarrow y)\approx t_2(x,y,y) \\
           & \hspace*{6mm}\text{by (5)}, \\ 
t_2(x,x,y) & \approx(x\vee y)\wedge(x\rightarrow y)\approx y\approx t_3(x,x,y)\text{ by (2)}, \\
t_1(x,y,x) & \approx(x\vee x)\wedge(y\vee x)\approx x\wedge(x\vee y)\approx x, \\
t_2(x,y,x) & \approx(x\vee x)\wedge(y\rightarrow x)\approx x\wedge(y\rightarrow x)\approx x\text{ by (1)}
\end{align*}
and hence $\mathcal V$ is congruence distributive. \\
Finally, if we put
\begin{align*}
t_1(x,y) & :=x\rightarrow y, \\
t_2(x,y) & :=y\rightarrow x
\end{align*}
then $t_1(x,y)=t_2(x,y)=1$ is equivalent to $x\leq y$ and $y\leq x$ according to Theorem~\ref{th4} and (5), i.e.\ to $x=y$ and hence $\mathcal V$ is weakly regular.
\end{proof}

\section{Algebras assigned to relatively residuated join- \\
-semilattices}

Similarly as for non-classical implication semilattices, we are going to show that also relatively residuated join-semilattices can be converted into algebras with everywhere defined operations by using of a certain ternary term $q(x,y,z)$ which replaces the partial operation $\odot$.

\begin{definition}\label{def4}
An {\rm R}-algebra is an algebra $(R,\vee,q,\rightarrow,1)$ of type $(2,3,2,0)$ such that $(R,\vee,1)$ is a join-semilattice with top element $1$ and the following conditions are satisfies for all $x,y,z,u\in R$:
\begin{enumerate}
\item[{\rm(20)}] $z\leq q(x,y,z)$,
\item[{\rm(21)}] $q(z\vee u\vee x,z\vee u\vee y,z)\approx q(z\vee u\vee x,z\vee u\vee y,z\vee u)$,
\item[{\rm(22)}] $q(x,1,x)\approx q(1,x,x)\approx x$,
\item[{\rm(23)}] $q(x,y,z)\approx q(y,x,z)$,
\item[{\rm(24)}] $q(q(x,y,u),z,u)\approx q(x,q(y,z,u),u)$,
\item[{\rm(25)}] $q(x,z,u)\leq q(x\vee y,z,u)$,
\item[{\rm(26)}] $x\vee z\leq y\rightarrow\big(q(x,y,z)\vee z\big)$,
\item[{\rm(27)}] $x\leq y\rightarrow x$,
\item[{\rm(28)}] $q(x,x\rightarrow y,y)\leq y$,
\item[{\rm(29)}] $q(x,y,z)\approx q(x\vee z,y\vee z,z)$,
\item[{\rm(30)}] $(x\vee y)\rightarrow y\approx x\rightarrow y$.
\end{enumerate}
\end{definition}

The following result is an immediate consequence of Definition~\ref{def4}.

\begin{corollary}
The class $\mathcal W$ of {\rm R}-algebras forms a variety.
\end{corollary}

Considering a relatively residuated join-semilattice as introduced in Definition~\ref{def2}, we can define a ternary operation $q(x,y,z)$ as follows: $q(x,y,z):=(x\vee z)\odot(y\vee z)$. For every relatively residuated join-semilattice $\mathbf R=(R,\vee,\odot,\rightarrow,1)$ put $\mathbb B(\mathbf R):=(R,\vee,q,\rightarrow,1)$.

\begin{theorem}\label{th7}
Let $\mathbf R=(R,\vee,\odot,\rightarrow,1)$ be a relatively residuated join-semilattice. Then $\mathbb B(\mathbf R)$ is an {\rm R}-algebra.
\end{theorem}

\begin{proof}
Let $\mathbb B(\mathbf R)=(R,\vee,q,\rightarrow,1)$ and $a,b,c,d\in R$. We have
\begin{enumerate}
\item[(20)] $c\leq(a\vee c)\odot(b\vee c)=q(a,b,c)$.
\item[(21)] $q(c\vee d\vee a,c\vee d\vee b,c)=(c\vee d\vee a)\odot(c\vee d\vee b)=q(c\vee d\vee a,c\vee d\vee b,c\vee d)$
\item[(22)] $q(a,1,a)=(a\vee a)\odot(1\vee a)=a\odot1=a$ and $q(1,a,a)=(1\vee a)\odot(a\vee a)=1\odot a=a$.
\item[(23)] $(a,b,c)=(a\vee c)\odot(b\vee c)=(b\vee c)\odot(a\vee c)=q(b,a,c)$.
\item[(24)] $q(q(a,b,d)c,d)=\Big(\big((a\vee d)\odot(b\vee d)\big)\vee d\Big)\odot(c\vee d)=\big((a\vee d)\odot(b\vee d)\big)\odot(c\vee d)=(a\vee d)\odot\big((b\vee d)\odot(c\vee d)\big)=(a\vee d)\odot\Big(\big((b\vee d)\odot(c\vee d)\big)\vee d\Big)$.
\item[(25)] $q(a,c,d)=(a\vee d)\odot(c\vee d)\leq\big((a\vee b)\vee d)\odot(c\vee d)=q(a\vee b,c,d)$.
\item[(26)] According to (17) we have
\[
a\vee c\leq b\rightarrow\Big(\big((a\vee c)\odot(b\vee c)\big)\vee c\Big)=b\rightarrow\big(q(a,b,c)\vee c\big).
\]
\item[(27)] coincides with (18).
\item[(28)] According to (18) and (19) we have
\[
q(a,a\rightarrow b,b)=(a\vee b)\odot\big((a\rightarrow b)\vee b\big)=(a\vee b)\odot(a\rightarrow b)\leq b.
\]
\item[(29)] $q(a,b,c)=(a\vee c)\odot(b\vee c)=\big((a\vee c)\vee c\big)\odot\big((b\vee c)\vee c\big)=q(a\vee c,b\vee c,c)$.
\item[(30)] coincides with (16).
\end{enumerate}
\end{proof}

For every {\rm R}-algebra $\mathbf R=(R,\vee,q,\rightarrow,1)$ put $\mathbb Q(\mathbf R):=(R,\vee,\odot,\rightarrow,1)$ where $x\odot y:=q(x,y,z)$ for all $x,y,z\in R$ with $z\leq x,y$. That $\mathbb Q(\mathbf R)$ is well-defined follows from the next theorem.

\begin{theorem}\label{th8}
Let $\mathbf R=(R,\vee,q,\rightarrow,1)$ be an {\rm R}-algebra. Then $\mathbb Q(\mathbf R)$ is a well-defined relatively residuated join-semilattice.
\end{theorem}

\begin{proof}
Let $\mathbb Q(\mathbf R)=(R,\vee,\odot,\rightarrow,1)$ and $a,b,c,d\in R$. Then $\odot$ is well-defined since in case $c,d\leq a,b$ we have
\[
q(a,b,c)=q(a\vee c\vee d,b\vee c\vee d,c)=q(a\vee c\vee d,b\vee c\vee d,c\vee d)=q(a\vee c\vee d,b\vee c\vee d,d)=q(a,b,d)
\]
by (21). Then because of (20), $c\leq q(a,b,c)=a\odot b$ whenever $c\leq a,b$.
\begin{enumerate}
\item[(10)] follows from the definition of an R-algebra.
\item[(11)] $a\odot1=q(a,1,a)=a$ and $1\odot a=q(1,a,a)=a$ by (22).
\item[(12)] Because of (23) we have $a\odot b=q(a,b,c)=q(b,a,c)=b\odot a$ whenever $c\leq a,b$.
\item[(13)] Because of (24) we have $(a\odot b)\odot c=q(q(a,b,d),c,d)=q(a,q(b,c,d),d)=a\odot(b\odot c)$ whenever $d\leq a,b,c$.
\item[(14)] Because of (25) we have $a\odot c=q(a,c,d)\leq q(a\vee b,c,d)=q(b,c,d)=b\odot c$ whenever $d\leq a\leq b$ and $d\leq c$.
\item[(16)] coincides with (30).
\item[(17)] Because of (26) and (29) we have
\[
a\vee c\leq b\rightarrow\big(q(a,b,c)\vee c\big)=b\rightarrow\big(q(a\vee c,b\vee c,c)\vee c\big)=b\rightarrow\big((a\vee c)\odot(b\vee c)\big)\vee c.
\]
\item[(18)] coincides with (27).
\item[(19)] Because of (27), (29) and (28) we have
\[
(a\vee b)\odot(a\rightarrow b)=(a\vee b)\odot\big((a\rightarrow b)\vee b\big)=q\big(a\vee b,(a\rightarrow b)\vee b,b)=q(a,a\rightarrow b,b)\leq b.
\]
\end{enumerate}
\end{proof}

\begin{theorem}\label{th9}
The correspondence described in Theorems~\ref{th7} and \ref{th8} is one-to-one.
\end{theorem}

\begin{proof}
If $\mathbf R=(R,\vee,\odot,\rightarrow,1)$ is a relatively residuated join-semilattice and
\begin{align*}
                   \mathbb B(\mathbf R) & =(R,\vee,q,\rightarrow,1), \\
\mathbb Q\big(\mathbb B(\mathbf R)\big) & =(R,\vee,\otimes,\rightarrow,1)
\end{align*}
then $x\otimes y=q(x,y,z)=(x\vee z)\odot(y\vee z)=x\odot y$ for all $x,y,z\in I$ with $z\leq x,y$ and hence $\mathbb J\big(\mathbb A(\mathbf I)\big)=\mathbf I$. If, conversely, $\mathbf R=(R,\vee,q,\rightarrow,1)$ is an R-algebra and
\begin{align*}
                   \mathbb Q(\mathbf R) & =(R,\vee,\odot,\rightarrow,1), \\
\mathbb B\big(\mathbb Q(\mathbf R)\big) & =(R,\vee,Q,\rightarrow,1)
\end{align*}
then $Q(x,y,z)\approx(x\vee z)\odot(y\vee z)\approx q(x\vee z,y\vee z,z)\approx q(x,y,z)$ according to (29) and hence $\mathbb B\big(\mathbb Q(\mathbf R)\big)=\mathbf R$.
\end{proof}

By Theorems~\ref{th7}, \ref{th8} and \ref{th9} the variety $\mathcal W$ of R-algebras is equivalent to the class of relatively residuated join-semilattices.

\begin{theorem}
The variety $\mathcal W$ of {\rm R}-algebras is weakly regular and its subvariety determined by the additional identity $q(x,x\rightarrow y,y)\approx y$ is $3$-permutable.
\end{theorem}

\begin{proof}
The proof is analogous to that of Theorem~\ref{th12} and uses also (i) of Theorem~\ref{th6}. Since $y\leq x\rightarrow y$ according to (18), the identity $q(x,x\rightarrow y,y)\approx y$ is equivalent to the identity
\[
(x\vee y)\odot(x\rightarrow y)\approx y
\]
which is nothing else than divisibility of the corresponding relatively residuated join-semilattice.
\end{proof}

{\bf Declaration of competing interest}

The authors declare that they have no known competing financial interests or personal relationships that could have appeared to influence the work reported in this paper.

Authors' addresses:

Ivan Chajda \\
Palack\'y University Olomouc \\
Faculty of Science \\
Department of Algebra and Geometry \\
17.\ listopadu 12 \\
771 46 Olomouc \\
Czech Republic \\
ivan.chajda@upol.cz

Helmut L\"anger \\
TU Wien \\
Faculty of Mathematics and Geoinformation \\
Institute of Discrete Mathematics and Geometry \\
Wiedner Hauptstra\ss e 8-10 \\
1040 Vienna \\
Austria, and \\
Palack\'y University Olomouc \\
Faculty of Science \\
Department of Algebra and Geometry \\
17.\ listopadu 12 \\
771 46 Olomouc \\
Czech Republic \\
helmut.laenger@tuwien.ac.at
\end{document}